\documentclass{amsart}

\usepackage{amsmath,amsthm,amssymb, epsf, graphics,  color, epsfig}
\usepackage{psfrag}
\usepackage{hyperref}
\numberwithin{equation}{section}
\newtheorem {theorem}{Theorem}
\newtheorem  {lemma}{Lemma}[section]

\newtheorem{definition}[lemma]{Definition}

\newcommand{\Sl}{\ensuremath{\mathrm{SL}(2,\mathbb R)}}
\newcommand{\Slz}{\ensuremath{\mathrm{SL}(2,\mathbb Z)}}

\newcommand{\gsl}{\Gamma\backslash\Sl}
\newcommand{\tf}{\tilde{f}}
\newcommand{\gh}{\Gamma\backslash\mathbb H}

\numberwithin{equation}{section}

\newcommand{\sga}{\sigma_{\mathfrak a}}

\newcommand{\re}{\textrm{Re}}

\newcommand{\one}{\frac{1}{2}}

\newcommand{\cq}{\chi_q}
\newcommand{\tcq}{\tau(\cq)}
\newcommand{\tcqb}{\overline{\tau(\overline{\cq})}}

\newcommand{\Sgao}{\sigma_{a_1}}
\newcommand{\Sgat}{\sigma_{a_2}}
\newcommand{\starsum}{\sideset{}{^*}\sum}
\begin{document}
\title{A fast algorithm to compute $L(1/2,f\times \cq)$}
\author{Pankaj Vishe}
\thanks{The author is supported by the G\"oran Gustafsson Foundation (KVA).}
\begin{abstract} Let $f$ be a fixed (holomorphic or Maass) modular cusp form. Let $\cq$ be a
Dirichlet character mod $q$. We describe a fast algorithm that computes the
value
$L(1/2,f\times\chi_q)$ up to any specified precision. In the case when $q$ is smooth or highly
composite integer, the time complexity of the algorithm is given by $O( 1+|q|^{5/6+o(1)})$.
\end{abstract}
\maketitle
\setcounter{tocdepth}{1}
\tableofcontents
\section{Introduction}
Let $\Gamma=\Slz$. Let $f$ be a fixed (holomorphic or Maass) cusp form on $\gh$. Let $\chi_q$ be a
Dirichlet character on $\mathbb Z/q\mathbb Z$. In this paper we consider the problem of computing
central values of $L$-function corresponding to $f$ twisted by $\cq$. The main result can be
summarized
as:

\begin{theorem}
 \label{main theorem 3}
Let $q,M,N$ be positive integers such that $q=MN$, where $M\leq N$, $M=M_1M_2 $  such that $M_1|N$
and $(M_2,N)=1 $. Let $f$ be a modular (holomorphic or Maass) form on $\gh$, $s\, \in \, \mathbb H$
and $\chi_q $ a Dirichlet character on $\mathbb Z/q\mathbb Z $. Let $\gamma,\epsilon $ be any
positive reals. Let $$E=\min\{M^{5}+N,q\}. $$ Then we can compute $L(s,f\times \chi_q) $ up
to an error of $O(q^{-\gamma}) $ in time $O(E^{1+\epsilon})$. The constants involved in $O$ are
polynomial in $(1+\gamma)/\epsilon$.
\end{theorem}

This method gives us a positive time saving if $q$ has a factor less than $q^{1/5} $. The
maximum
saving of size $O(q^{1/6}) $ can be obtained if $q$ has a suitable factor of size $\approx q^{1/6}
$. In particular, we can get a saving of size $O(q^{1/6})$ for a ``smooth'' or ``highly composite''
integer $q$. Note that for these choices of $q$, the algorithm is considerably faster than the
$O(q^{1+o(1)})$ complexity ``approximate functional equation'' based algorithms.

\subsection{Model of computation} We will use the real number (infinite precision) model of
computation that uses real numbers with error free arithmetic having cost as unit cost per
operation. An operation here means addition, subtraction, division, multiplication, evaluation of
logarithm (of a complex number $z$ such that $|\arg (z)|<\pi$) and exponential of a complex number.

Our algorithm will work if we work with numbers specified by $O(\log q)$ bits. This will at most add
a power of $\log q$ in the time complexity of the algorithm. We refer the readers to \cite[Chapter
8]{traub} and \cite{traub1} for more details about the real number model of computation.

\subsection{Historical background and applications}

The problem of ``computing'' values of the zeta function effectively goes as far back as Riemann.
Riemann used the Riemann Siegel formula to compute values of the zeta function and verify the
Riemann hypothesis for first few zeroes. The Riemann Siegel formula writes $\zeta(1/2+iT)$ as a main
sum of length $O(T^{1/2})$ plus a small easily ``computable'' error. Subsequent improvements for
the rapid evaluation of zeta are considered in \cite{odlyzko},
 \cite{schon}, \cite{hiary}, \cite{hiary1}, \cite{turing}, \cite{berry},
\cite{rubinstein} and\cite{arias} \textit{et al.} The current fastest
algorithm for evaluating $\zeta(1/2+iT)$ for a single value of $T$ is due to Hiary (time
complexity $O(T^{4/13+o(1)})$, see \cite{hiary1}). 

The next natural problem to consider is computing $L(1/2+iT,\chi_q)$, where $\chi_q$ is a Dirichlet
character modulo an integer $q$. A $O(T^{1/3+o(1)}q^{1/3+o(1)})$ algorithm for highly composite $q$,
is given by Hiary in \cite{hiary2}. In this algorithm, the rapid computation of $L(1/2,\chi_q)$ is
essentially reduced to the problem of fast computations of character sums
$\sum_{k=k_0}^{k_0+M}\chi_q(k)$, for any $k_0$ and for $M$, a small power of $q$. In case
of highly composite $q$, one exploits highly repetitive nature of $\chi_q$ to get a fast way of
computing these character sums. The problem of computing $L(1/2,\chi_q)$ for an ``almost prime''
$q$ seems to be rather difficult.

In the case of higher rank $L$-functions, the analogue of the Riemann Siegel formula is given by the
approximate functional equation. A detailed description of ``the approximate functional equation''
based algorithms is given by Rubinstein in \cite{rubinstein}. In the case of $L$-function
associated to a modular (holomorphic or Maass) form, these algorithms have $O(T^{1+o(1)})$ time
complexity. The algorithms for rapid computation
of the $\mathrm{GL}(1)$ $L$- functions unfortunately do not readily generalize to the higher
rank cases, due to the complicated main sum in the smooth approximate functional equation.
A geometric way for computing $L(f,1/2+iT)$ in time $O(T^{7/8+o(1)})$, for a modular form $f$ is
given in \cite{vishe1}.

In this paper we consider the higher rank problem corresponding to \cite{hiary2}, \textit{i.e}
computing $L(1/2,f\times \chi_q)$. We give a $O(q^{5/6+o(1)})$ complexity algorithm for a
``smooth'' or ``highly composite'' $q$. Our algorithm in theorem \ref{main theorem 3} is the first
known
improvement of the approximate functional equation based algorithms in in $\mathrm{GL}(2)\times
\mathrm{GL}(1)$ setting.

Computing values of $L$-functions on the critical line has various applications in number theory. It
can be used to verify the \textit{Generalized Riemann Hypothesis} numerically. It has also been used
to connect the distribution of values of $L$-functions on the critical line to the distributions of
eigenvalues of unitary random matrices via the recent random matrix theory conjectures. 

The problem of computing $L$-functions is closely related to the problem of finding subconvexity
bounds for the $L$-functions. More generally, improving the ``square root of analytic conductor
bounds'' coming from the approximate functional equation is of great interest to analytic number
theorists.

In the present paper, we have only considered the $\mathrm{GL}_2\times \mathrm{GL}_1$ case. It will
be of great interest
to generalize the method in this paper to higher rank $L$-functions. Integral representations
for $\mathrm{GL}_2\times \mathrm{GL}_1$ $L$-functions in the number field setting is given by Sarnak
in
\cite[section 11.4]{venkatesh}. Our method thus could also generalize for
$L$-functions $L(1/2,f\times\chi)$, in the number field setting. A more interesting problem will be
to generalize our
technique in the $\mathrm{GL}_n\times\mathrm{GL}_{n-1}$ setting, where the subconvexity bounds are
also not known.

\subsection{Outline of the proof}
Our algorithm starts with writing $L(1/2, f\times\chi_q)$ essentially as a sum
\[\sum_{j=0}^{q-1}f(j/q+i/q)\chi_q(j),\] on the $q^{\mathrm{th}}$ Hecke orbit of $i$. The results
in
this paper are closely related to the work of
Venkatesh \cite[section 6]{venkatesh}, where he used the equidistribution of the points
$\{j/q+i/q:0\leq j\leq q-1\}$ in $\gh$, to get subconvexity
bounds for the $L$-functions. We however use the fact that for a composite $q=MN$, we have the
following decomposition:
$$\{j/q+i/q:0\leq j\leq q-1\}=\cup_{j=0}^{N-1} \{(j+kN)/q+i/q:0\leq k\leq
M-1\}.$$ Each arithmetic progression
$\{(j+kN)/q+i/q:0\leq k\leq
M-1\}$ can be viewed as part of the $M^{\mathrm{th}}$ Hecke orbit of the point $v_j=j/N+i/N$.

Let us consider the case $(M,N)=1$. We use
``well behaved nature'' of $\chi_q$ on these arithmetic progressions to convert the problem into the
problem of computing sums
$$S_j=\sum_{k}a_{k,c_j} f(v_j(k)).$$
Here, $0\leq c_j\leq M-1$ is such that $j\equiv c_j\bmod{M}$ and $v_j(k)$ denote the $k^{th}$
point in the $M^{th}$ Hecke orbit of the point $v_j$. Here $\{a_{k.l},0\leq k,l\leq M-1\}$ are $M^2$
precomputable constants.

We then ``reduce'' the points $\{v_j\}$ to points $\{x_j\}$ in a fixed approximate
fundamental domain. Let $a_j$ be in $\Slz$ which maps $v_j$ to $x_j$. Notice that even though in
$\gsl$, the Hecke orbits of points $v_j$ and $x_j$ are the same sets, the action of $a_j$ permutes
them. This means that $v_j(k)=a_j^{-1}x_j(k)$ is not necessarily equal to $x_j(k) $. However for
each $j$, $v_j(k)=x_j(\sigma_j(k))$, for some permutation $\sigma_j$ of the Hecke orbit. The problem
is now equivalent to computing $$S_j=\sum_k a_{k,c_j} f(x_j(\sigma_j(k)))=\sum_k
a_{\sigma_j^{-1}(k),c_j} f(x_j(k)).$$

Apriori, there can be $M!$ different permutations of the Hecke orbit. However lemma
\ref{permutations} shows that the total different number of permutations of the $M^{th}$ Hecke orbit
due to the right action of $\Slz$ is at most $O(M^3)$. {\em i.e.} given $M$, there exists a set of
permutations $\{\beta_1,...,\beta_{M^3}\}$, such that for any $0\leq j\leq N-1$,
$\sigma_j=\beta_k$ for some $0\leq k\leq M^3-1$.  

Let $\{r_1,...,r_{M^4}\}$ be the functions defined
on the $M^{th}$ Hecke orbit by $$r_{j_1M^3+j_2}=a_{\beta_{j_2}(k),j_1}, $$ where $0\leq
j_1\leq M-1$ and $0\leq j_2\leq M^3-1$. This implies that every $j$, there exists an integer
$l_j\leq M^4$ such that $a_{\sigma_j^{-1}(k),c_j}=r_{l_j}(k) $.

We next use the fact that Hecke
orbits of close enough points in
the $\gsl$ remain close. We then sort the
points $\{x_j\}$ into sets $K_{1},K_{2},..$ such that the points in each set
are very close to one other. The number of such required sets is $O(q^{3\epsilon}) $. We choose a
fixed representative $x_{n_j}$ from each $K_j$. 

Let us consider $K_1$ for example. Let $x_{j_0}\in K_1$ be any point. We now need to evaluate
$S_{j_0}$ for all $j_0\in K_1$. We use a power series expansion around the point
$x_{n_1}(k)$ to compute $f(x_{j_0}(k)) $ for each $x_{j_0}\in K_1$, to get 
$$S_{j_0}=\sum_{k=0}^{d}
c_{k,j_0}J(k,r_{l_{j_0}},x_{n_1}).$$
Here $d=O(1)$, $c_{k,j_0}$ are precomputable constants and $J(k,r_l,x_{n_1})$ can be computed for
each
$0\leq l\leq M^4-1$ and for each $k$ using $O(M)$ steps. Thus, if we compute the coefficients
$c_{k,j_0}$ and the sums $J(k,r_{l},x_{n_1})$ for all $x_{j_0}\in K_1$, for all $k\leq d $ and for
all
$0\leq l\leq M^4-1 $, then we
can compute
all the sums $S_j$ for all $j\in K_{1}$ ``in parallel'' in further $O(1)$ time. We can therefore
compute $S_j$ for all $j$ in $O((M^5+N)^{\epsilon})$ time.

We thus get a fast way of numerically computing $L(1/2,f\times\chi_q)$, up to any given precision.  

\subsection{Outline of the paper}
A brief account of the notations used in this paper is given in section \ref{notation}. The
algorithm uses a type of ``geometric approximate functional equation''. It is discussed in detail in
section \ref{geometric}. A detailed proof of theorem \ref{main theorem 3} is given in section
\ref{q-aspect}. Lemmas \ref{epsilonchi} and \ref{permutations} deal with well behaved nature of
Hecke orbits of
nearby points in $\mathbb H$ and well behaved nature of $\chi_q $ on the ``arithmetic
progressions'' respectively. They are proved in section \ref{epsilonpower}.

\subsection{Acknowledgements}
I would like to thank Akshay Venkatesh for suggesting
me this problem. This project would not have been possible without his support.
Majority of the work in the paper was carried out during my PhD at the Courant
Institute and Stanford University, and also during my stay EPFL and MPIM, Bonn.
I am very thankful to these institutions for allowing me to visit and carry out my
research

\section{Notation and preliminaries}
\label{notation}
Throughout, let $\Gamma=\Slz$. Let $f$ be a holomorphic or Maass cusp form on $\gh$.

Let $s$ be a fixed point in $\mathbb H$. Throughout, let $q$ be a positive integer. Let $\cq$ be a
Dirichlet character on $\mathbb Z/q\mathbb Z$. Let $\gamma,\epsilon$ be any given
positive numbers, independent of $q$. In practice, $\gamma$ will be taken to be $O(1)$ and
$\epsilon$ will be a small positive number.

Given a Dirichlet character $\cq$, The Gauss sum $\tcq $ is defined by
\begin{equation}\label{gauss}\tau(\cq)=\sum_{k=0}^{q-1}\cq(k)e(k/q). \end{equation}
We will denote the set of nonnegative integers by $\mathbb Z_+ $ and the set of nonnegative real
numbers by $\mathbb R_+$.

We will use the symbol $\ll$ as is standard in analytic number theory: namely, $A\ll B$ means that
there exists a positive constant $c$ such that $A\leq cB $. These constants will always be
independent of the choice of $T$.

We will use the following special matrices in $\Sl$ throughout the paper:
\begin{align}
n(t)=\left(\begin {array}{cc} 1&t\\0&1\\ \end{array} \right), a(y)=\left(\begin {array}{cc}
e^{y/2}&0\\0&e^{-y/2}\\ \end{array} \right),K(\theta)=\left(\begin {array}{cc} \cos \theta&\sin
\theta\\-\sin \theta&\cos \theta\\ \end{array} \right).
\end{align}

$e(x)$ will be used to denote $\exp(2\pi i x) $.

In this paper, for simplicity let us assume that $f$ is either holomorphic or an even Maass form.
The algorithms will be analogous for the odd Maass cusp form case. For an even Maass cusp form, we
will use the following power series expansion :

\begin{equation}
\label{Whittaker eqeven}
f( z) =\sum_{n> 0}\hat{f}(n)W_{r}(nz).
\end{equation}
Here $W_{r}(x+iy)=2\sqrt{y}K_{ir}(2\pi y) \cos (2\pi x) $. The explicit Fourier expansion for the
holomorphic cusp forms is given by
\begin{equation}
 \label{Whittaker hol}
f(z) =\sum_{n> 0}\hat{f}(n)e(nz).
\end{equation}

For a cusp form of weight $k$ (for the case of Maass forms $k$ is assumed to be 0), the
corresponding twisted $L$- function is defined by:

\begin{definition}
 \label{L chi function}
$L(s,f\times\cq)=\sum_{n=1}^\infty \frac{\hat{f}(n)\chi_q(n)}{n^{s+(k-1)/2}}$.
\end{definition}

Given a cusp (Maass or holomorphic) form of weight $k$ on $\gh $, we will define a lift $\tf$ of $f$
to $\gsl$ by $\tf: \gsl\rightarrow \mathbb C$ such that \begin{equation}\label{tf}\tf(\left(\begin
{array}{cc}a&b\\c&d \end{array} \right))=(ci+d)^{-k}f(\frac{ai+b}{ci+d}). \end{equation}

\subsection{Real analytic functions on $\gsl$}
\label{real analytic notation}
We will use the same notation as \cite{vishe1} for real analytic functions on $\gsl$.

Let $x$ be an element of $\Sl$ and let $g$ be a function on $\gsl$, \textit{a priori} $g(x)$ does
not make sense but throughout we abuse the notation to define  $$g(x)=g(\Gamma x). $$ \textit{i.e.} 
$g(x)$ simply denotes the value of $g$ at the coset corresponding to $x$.

Let $\phi$ be the Iwasawa decomposition given by $$\phi:(t,y,\theta)\in\mathbb R\times \mathbb R
\times \mathbb R\rightarrow n(t)a(y)K(\theta).$$ Recall that $\phi$ restricted to the set $\mathbb
R\times \mathbb R \times (-\pi,\pi]$ gives a bijection with $\Sl$.
\begin{definition}
\label{U definition}
Given $\eta>0$, let $\mathfrak U_{\eta}=(-\eta,\eta)\times
(-\eta,\eta)\times (-\eta,\eta)$ and $U_\eta=\phi(\mathfrak U_\eta)\subset
\Sl$.
\end{definition}

Let us define the following notion of ``derivatives'' for smooth functions on $\gsl$:

\begin{definition}
\label{liederivative}
Let $g$ be a function on $\Sl$ and $x$ any point in $\Sl$. We define (wherever R.H.S. makes sense)
\begin{align*}
\frac{\partial}{\partial x_1} g(x)&=\frac{\partial}{\partial t}|_{t=0}g(xn(t));\\
\frac{\partial}{\partial x_2} g(x)&=\frac{\partial}{\partial t}\mid_{t=0}g(xa(t));\\
\frac{\partial}{\partial x_3} g(x)&=\frac{\partial}{\partial t}\mid_{t=0}g(xK(t)).
\end{align*}
\end{definition}
Sometimes, we will also use $\partial_i $ to denote $\frac{\partial}{\partial x_i} $.

Given $\beta=(\beta_1,\beta_2,\beta_3)$, let us define $\partial^\beta g(x)$ by \begin{equation}
\partial^\beta g(x)=\frac{\partial^{\beta_1}}{\partial
x_1^{\beta_1}}\frac{\partial^{\beta_2}}{\partial x_2^{\beta_2}}\frac{\partial^{\beta_3}}{\partial
x_3^{\beta_3}}g(x).
\label{partialbeta}
\end{equation}

For $\beta$ as above we will define $$\beta!=\beta_1!\beta_2!\beta_3!$$ and
$$|\beta|=|\beta_1|+|\beta_2|+|\beta_3|. $$
We now define the notion of real analyticity as follows:

A function $g$ on $\gsl$ will be called real analytic, if given any point $x$ in $\gsl$, there
exists a positive real number $r_x$ such that $g$ has a power series expansion given by

\begin{equation}
g(xn(t)a(y)K(\theta))=\sum_{\beta=(\beta_1,\beta_2,\beta_3)\in\mathbb Z_+^3}\frac{\partial^\beta
g(x)}{\beta!}t^{\beta_1}y^{\beta_2}\theta^{\beta_3}
\label{power}
\end{equation}
for every $(t,y,\theta)\in \mathfrak U_{r_x} $.

Let us use the following notation for the power series expansion.
\begin{definition}
\label{power series}
Let $y,x\in \Sl$ and $t,y,\theta$ be such that $y=xn(t)a(y)K(\theta) $ and
$(\beta_1,\beta_2,\beta_3)=\beta\in \mathbb Z_+^3$ define $$(y-x)^\beta
=t^{\beta_1}y^{\beta_2}\theta^{\beta_3} .$$
\end{definition}

Hence we can rewrite the Equation \eqref{power} as
$$g(y)=\sum_{\beta=(\beta_1,\beta_2,\beta_3),\beta\in \mathbb Z_+^3} \frac{\partial^\beta
g(x)}{\beta!}(y-x)^\beta.$$

Throughout, we will assume that for a cusp form $f$, all the derivatives of the lift $\tf$ 
are bounded uniformly on $\gsl$ by 1. In general it can be proved that given a cusp form $f$, there
exists $R$ such that $||\partial^\beta \tf||_\infty \ll R^{|\beta|}$, see \cite[section 8.2]{vishe}.
The case when $R>1$ can be dealt with analogously. The assumption that all derivatives are bounded
by 1, allows the proofs to be marginally simpler.

\subsection{Hecke orbits}
\label{heckeorbits}
Let $L$ be any positive integer and $x\in \Sl $, let $$T(L)=\{(m,k): m|L, 0\leq
k< L/m\}$$ and
 $$A(L,m,k)=\frac{1}{L^{1/2}}\left(\begin {array}{cc}m&k\\0&L/m \end{array} \right).$$ The $L^{th}$
Hecke orbit is given by left action of the cosets $\{\Gamma A(L,m,k): (m,k)\in T(L) \}$. In
particular, the $L^{th}$ Hecke orbit of
$x$ is given by $\{\Gamma A(L,m,k)x, (m,k)\in T(L) \}$, considered as a subset of $\gsl $. The Hecke
orbits generalize the notion of an ``arithmetic progression'' on $\Sl $. 

 It is well known that the right action of any element $a $ in $\Slz$ permutes the cosets $\{\Gamma
A(L,m,k): (m,k)\in T(L) \} $. {\em i.e} \[\{\Gamma A(L,m,k): (m,k)\in T(L)  \}=\{\Gamma
A(L,m,k)a: (m,k)\in T(L) \}\] for any $a\in \Slz$. Let
$\sga:T(L)\rightarrow T(L) $ be be the permutation defined by \[\Gamma A(L,m,k) a=\Gamma
A(L,\sga(m,k)).\]

Given any $\epsilon>0$, using the fact that the number of divisors of $M$ is at most
$O(M^{\epsilon})$, we get that the cardinality of $T(M)$ is at most $O(M^{1+\epsilon}) $. \textit{A
priori} there are $M^{1+\epsilon}!$ possible permutations on $T(L)$. However it is easy to prove the
following lemma (see section \ref{epsilonpower}):
\begin{lemma}
\label{permutations}
 Let $a_1$ and $a_2 $ are matrices in $\Slz$ such that $a_1\equiv a_2\bmod{L}$. Then the
corresponding permutations $\Sgao $ and $\Sgat$ are equal. In other words, for every $(m,k)\in
T(L)$, $\Sgao(m,k)=\Sgat(m,k) $.
\end{lemma}

Lemma \ref{permutations} implies that the number of possible permutations of the Hecke orbit
$\{\Gamma A(L,m,k)|(m,k)\in T(L)\} $ due to the right action of $\Slz $ are at most
$|\mathrm{SL}(2,\mathbb
Z/L\mathbb Z)| \leq L^3 $.

\subsection{Specifying $f$ and $\chi_q$}
We will follow the same assumptions for input of $f$ as in \cite{vishe1}. In particular, we will
assume that each value of $f$
(or $\tf$) or any of its derivative can be computed exactly in time $O(1)$. \footnote{ It can be
easily shown that given any $x$ and any fixed
$\gamma$, one can compute $\partial^\beta\tf(x)$ up to the error $O(q^{-\gamma})$ in  $O(q^{o(1)})$
time. Here the constant involved in $O$ is a polynomial in $|\beta|$ and $\gamma$. In this algorithm
we only compute values of $\partial^\beta \tf(x)$ for $|\beta|\ll 1$. Allowing $O(q^{o(1)})$ time
for each valuation of $\tf$ does not change the time complexity of the algorithm. See \cite[chapter
7]{vishe} for explicit details about it}

We will also assume that given any integer $n$, $\chi_q(n)$ can be computed in time $O(1)$. It can
be easily be shown that for $q=MN$, using and storing a precomputation of size $O(M+N)$, one can
compute any $\chi_q(n)$, in further $O(\log q)$ steps. The time complexity $O(M+N)$ in
precomputation does not change the asymptotics of the algorithm. Similarly, allowing $\log(q)$ time
for each valuation of $\chi_q$ only adds a multiple of $\log q$ to the time complexity of the
algorithm, which can be absorbed into the exponent $\epsilon.$ 

In practice, the Gauss sum $\tcq$ can be computed rather rapidly. A very simple $O(M^2+N)$ time
complexity algorithm can be found in \cite[section 8.6]{vishe}. As in the previous case, this does
not change the asymptotic time complexity of the algorithm. It is worth mentioning to the reader
that the algorithm in \cite[section 8.6]{vishe} is similar to the algorithm in the paper and can be
helpful in better understanding of the underlying idea behind it.

 \section{A geometric approximate functional equation for $L(s,f\times \chi_q) $}
\label{geometric}
Our algorithm will start with proving a ``geometric approximate
functional equation'' for $L(s,f\times \chi_q) $, given by \eqref{approxfechihol} and
\eqref{approxfechimaass}. The right hand side of \eqref{approxfechihol} and
\eqref{approxfechimaass} consists of sums of $q$ integrals. Each of these integrals is an integral
of a `nice' function on a geodesic of (hyperbolic) length $O(\log q) $. Therefore, using
\cite[proposition 8.1]{vishe1}, we can write each of these integrals (up to an error of
$O(q^{-\gamma}) $), as a sum of size $O(q^{\epsilon} ) $ terms. Adding all these sums
together, the right hand sides of  \eqref{approxfechihol} and \eqref{approxfechimaass} can be
written (up to an error of $O(q^{-\gamma}) $) a sum of size $O(q^{1+\epsilon}) $. The constants
involved in $O$ are polynomial in $(1+\gamma)/\epsilon $ and are independent of $q$.

Holomorphic case:
\begin{align}\label{holochi}\sum_{k=0}^{q-1}\chi_q(k)f(k/q+iy)&=\sum_{k=0}^{q-1}\chi_q(k)\sum_{n=1}^
{\infty}\hat{f}(n)e(nk/q)e(iny)\nonumber \\&=\sum_{n=1}^{\infty}\hat{f}(n)e(iny)\sum_{k=0}^{q-1}
\chi_q(k)e(nk/q )\nonumber\\&= \tau(\cq)\sum_{n=1}^{\infty}\hat{f}(n)\cq(n) e(iny).\end{align}
After taking the Mellin transform of \eqref{holochi}, we get
\begin{align}
 \label{holochi1}\sum_{k=0}^{q-1}\chi_q(k)&\int_0^\infty
f(k/q+iy)y^{s+(k-3)/2}dy\\&=\tau(\cq)\sum_{n=1}^{\infty}\hat{f}(n)\cq(n) \int_0^\infty
e(iny)y^{s+(k-3)/2}dy; \nonumber \\&=\tau(\cq)\sum_{n=1}^{\infty}\hat{f}(n)\cq(n) \int_0^\infty
\exp(-2\pi ny)y^{s+(k-3)/2}dy; \nonumber \\&=\frac{\tau(\cq)}{(2\pi)^{s+(k-1)/2}}L(s,f\times
\cq)\Gamma(s+(k-1)/2).\nonumber
\end{align}

Even Maass form case:
\begin{align}\sum_{k=0}^{q-1}\chi_q(k)&f(k/q+iy)\nonumber\\&=2\sum_{k=0}^{q-1}\chi_q(k)
\sum_{n=1}^{\infty}\hat{f}(n)\sqrt{ny}\cos(2\pi nk/q)K_{ir}(2\pi
ny)\nonumber\\&=2\sum_{n=1}^{\infty}\hat{f}(n)\sqrt{ny}K_{ir}(2\pi ny)\sum_{k=0}^{q-1}
\chi_q(k)\frac{e(nk/q)+e(-nk/q)}{2}\nonumber\\&= 2\frac{\tau(\cq)+\tcqb
}{2}\sum_{n=1}^{\infty}\hat{f}(n)\cq(n)\sqrt{ny} K_{ir}(2\pi ny)\nonumber\\&=\tcq
(1+\cq(-1))\sum_{n=1}^{\infty}\hat{f}(n)\cq(n)\sqrt{ny} K_{ir}(2\pi ny).\label{maasschi}\end{align}

Taking Mellin transform of \eqref{maasschi}, we get

\begin{align}
 \label{maasschi1}\sum_{k=0}^{q-1}\chi_q(k)&\int_0^\infty f(k/q+iy)y^{s-3/2}dy\\&=\tcq
(1+\cq(-1))\sum_{n=1}^{\infty}\hat{f}(n)\sqrt{n}\cq(n) \int_0^\infty K_{ir}(2\pi
ny)y^{s-1}dy; \nonumber \\&=\tcq \frac{1+\cq(-1)}{(2\pi)^{s}}L(s,f\times \cq)\Gamma\left(\frac{s+i
r}{2}\right)\Gamma\left(\frac{s-i r}{2}\right).\nonumber
\end{align}

$\tcq $ is a complex number with absolute value $q^{\one} $. Therefore we can use
\eqref{holochi1} to compute $L(s,f\times \cq) $ in the holomorphic case.

For an even Maass form $f$ however, if $\chi_q(-1)=-1
$, then the right hand side of \eqref{maasschi1} is zero. Therefore
it needs slightly different treatment. Recall the Fourier
expansion for $f$ given by

\begin{equation*}
 f(z) =\sum_{n> 0}\hat{f}(n)2\sqrt{y}K_{ir}(2\pi ny) \cos (2\pi nx).
\end{equation*}
This implies that $$\partial_x f(z)=-2\pi \sum_{n> 0}n\hat{f}(n)2\sqrt{y}K_{ir}(2\pi ny) \sin (2\pi
nx).$$ We use a similar method as before to get:
\begin{align}\sum_{k=0}^{q-1}\chi_q(k)&\partial_xf(k/q+iy)\nonumber\\&=-2\pi\sum_{k=0}^{q-1}
\chi_q(k)\sum_{n=1}^{\infty}n\sqrt{ny}\hat{f}(n)\sin(2\pi nk/q)K_{ir}(2\pi
ny)\nonumber\\&=-2\pi\sum_{n=1}^{\infty}\hat{f}(n)n^{3/2}\sqrt{y}K_{ir}(2\pi ny)\sum_{k=0}^{q-1}
\chi_q(k)\frac{e(nk/q)-e(-nk/q)}{2i}\nonumber\\&=
i\pi(\tau(\cq)-\tcqb)\sum_{n=1}^{\infty}\hat{f}(n)n^{3/2}\cq(n) \sqrt{y}K_{ir}(2\pi
ny)\nonumber\\&=i\pi\tcq (1-\cq(-1))\sum_{n=1}^{\infty}\hat{f}(n)n^{3/2}\cq(n) \sqrt{y}K_{ir}(2\pi
ny).\label{maasschineg}\end{align}
Taking Mellin transform of \eqref{maasschineg}, we get

\begin{align}
 \label{maasschineq1}&\sum_{k=0}^{q-1}\chi_q(k)\int_0^\infty \partial_xf(k/q+iy)y^{s-1/2}dy\\&=i\pi
\tcq (1-\cq(-1))\sum_{n=1}^{\infty}\hat{f}(n)\cq(n)n^{3/2} \int_0^\infty K_{ir}(2\pi ny)y^{s}dy
\nonumber \\&=i\pi\tcq \frac{1-\cq(-1)}{(2\pi)^{s+1}}L(s,f\times \cq)\Gamma\left(\frac{s+1+i
r}{2}\right)\Gamma\left(\frac{s+1-i r}{2}\right).\nonumber
\end{align}

Notice that $\eqref{maasschineq1}$ is analogous to \eqref{maasschi1}. The algorithm to
compute\linebreak
$\sum_{k=0}^{q-1}\chi_q(k)\int_0^\infty \partial_xf(k/q+iy)y^{s-1/2}dy $ is completely
analogous to the algorithm  to compute $\sum_{k=0}^{q-1}\chi_q(k)\int_0^\infty f(k/q+iy)y^{s-3/2}dy
$.
Hence throughout the rest of the paper, we will assume that $\cq(1)=\cq(-1)=1$.
Using the automorphy of $f$, we get the following lemma (analogous to \cite[Lemma 3.1]{vishe1}).

Note that a similar treatment will give us the corresponding ``geometric approximate functional
equations'' for odd Maass forms.

\begin{lemma}
\label{cut the sum 1}
Given any cusp form $f$ (of weight $k$) on $\gsl$, positive integers $n,q$ such that $n<q $,
$s\;\in \;\mathbb H$, a positive real $\gamma$ and for any $c>2$,
 \begin{equation}
  \label{decayexp}
    \int_0^\infty f(n/q+iy)y^s dy= \int_{q^{-c}}^{q^c}f(n/q+iy)y^s
dy+O(q^{-\gamma}).
 \end{equation}
The constant involved in $O$ is independent of $q$ and $c$.
\end{lemma}
\begin{proof}
 We begin by noting that it is enough to prove the lemma when $n$ is coprime to $q$. Let us use the
exponential
decay of $f$ at $i\infty$, to get for $y\geq 1$, 
\begin{equation}
|f( n/q+i y) |\ll \exp(-\pi y).
\label{n/q bound}
\end{equation}

This implies that we can choose a constant $c_1$, independent of $n,q $ such that for
every $y\geq c_1\log q $, we have \begin{equation}\label{upper}|f(n/q+i y)|\ll
q^{-|s|-\gamma-2}y^{-2}
.\end{equation}

Let $n',n''$ such that $0\leq n'<q $ and $nn'-qn''=1 $. The the action of $g=\left(\begin
{array}{cc}n'&-n''\\-q&n \end{array} \right)$ on $ \mathbb H$ maps $n/q$ to infinity. Using the
automorphy of $f$ with respect to the action of $g$, we get

\begin{align}
 \label{autochi}
f(n/q+iy)&=(-q(n/q+iy)+n))^{-k}f(\frac{n'(n/q+iy)-n'' )}{-q(n/q+iy )+n} )\\&=
(-qiy)^{-k}f(\frac{1/q+in'y}{-qyi} )\nonumber\\&=(-qiy)^{-k}f(-n'/q+\frac{i}{q^2y}). \nonumber
\end{align}

We use the exponential decay of $f$ at infinity to get a
constant $c_2$, independent of $n,q$ such that for every $y\leq \frac{1}{c_2 q^2 \log q} $ , we have
\begin{equation}\label{lower}|f(n/q+iy) |\ll q^{-\gamma-|s|-2}.\end{equation}

The equations \eqref{upper} and \eqref{lower} give us the result.

\end{proof}

Lemma \ref{cut the sum 1} implies that for any $c>2$, 
\begin{align}
 \label{approxfechihol}
&\frac{\tau(\cq)}{(2\pi)^{s+(k-1)/2}}L(s,f\times
\cq)\Gamma(s+(k-1)/2)\\&=\sum_{j=0}^{q-1}\chi_q(j)\int_{q^{-c}}^{q^c}
f(j/q+iy)y^{s+(k-3)/2}dy+O(q^{-\gamma}).\nonumber
\end{align}

and similarly we get that given any $q,\gamma>0 $, a Maass cusp form $f$, and $c>2$, we have

\begin{align}
 \label{approxfechimaass}
&\tau(\cq)\frac{1+\chi_q(-1)}{4(\pi)^{s}}L(s,f\times \cq)\Gamma\left(\frac{s+i
r}{2}\right)\Gamma\left(\frac{s-i r}{2}\right)\\&=\sum_{k=0}^{q-1}\chi_q(k)\int_{q^{-c}}^{q^c}
f(k/q+iy)y^{s-1}dy+O(q^{-\gamma}).\nonumber
\end{align}
 The equations \eqref{approxfechihol} and \eqref{approxfechimaass} denote ``geometric approximate
functional equations'' to compute $L(s,f\times \cq) $ in the holomorphic and Maass form case
respectively. Notice that the integrals on the right hand side of \eqref{approxfechihol} and
\eqref{approxfechimaass} are over hyperbolic curves of length $\ll \log q$, therefore they can be
computed in $O(q^{o(1)}) $ time. In the following theorem, we discretize the integrals in these
equations to convert the problem
of computing the sum on the right hand side of \eqref{approxfechihol}/\eqref{approxfechimaass} to
the problem of computing the sum $S=\sum_{j=0}^{q-1}\chi_q(j)\partial_2^l\tf(n(j/q)a(t))$ for any
$t\in [-c\log q,c\log q ]$. Here $\partial_2$ denotes $\frac{\partial}{\partial x_2}$.

\begin{lemma}
 \label{integral to sum fe}
Let $f$ be a modular (holomorphic or Maass ) cusp form on $\gh $, and $\chi_q $ be a Dirichlet
character modulo $ q$, $s$ be any complex number. Let
$\gamma,\epsilon $ be any
positive reals, then there exists a positive integer $N'=O((1+\gamma)/\epsilon) $ 
such that if for
any $|t|\ll \log q $ and for any $0\leq l\leq N' $, we can compute
$\sum_{j=0}^{q-1}\chi_q(j)\partial_2^l\tf(n(j/q)a(t)) $ up to a maximum error $O(q^{-\gamma}) $ in
time $D(q) $, then we can compute $\tcq L(s,f\times \chi_q) $ using $O(D(q) q^{\epsilon}) $
operations. The constants in $O$ are polynomial in $(1+\gamma)/\epsilon$.
\end{lemma}
\begin{proof}
Let us start with the ``geometric approximate functional equations''
\linebreak\eqref{approxfechihol}/ \eqref{approxfechimaass} for holomorphic/Maass case respectively.
We use the lift $\tf$  defined in \eqref{tf} to get $\tf(n(x)a(\log
y))=y^{k/2}f(x+iy)$ to rewite \eqref{approxfechihol} as:

\begin{align*}
\frac{\tau(\cq)}{(2\pi)^{s+(k-1)/2}}&L(s,f\times
\cq)\Gamma(s+(k-1)/2)\\&=\sum_{j=0}^{q-1}\chi_q(j)\int_{q^{-c}}^{q^c}
\tf(n(j/q)a(\log y))y^{s-3/2}dy+O(q^{-\gamma}).
\end{align*}
Notice that the above equation will be valid for any $c>2$. Substitute $\log y=t $ in the above
equation to get that 
\begin{align*}
\frac{\tau(\cq)}{(2\pi)^{s+(k-1)/2}}&L(s,f\times
\cq)\Gamma(s+(k-1)/2)\\&=\sum_{j=0}^{q-1}\chi_q(j)\int_{-c\log q}^{c\log q}
\tf(n(j/q)a(t))e^{t(s-1/2)}dt+O(q^{-\gamma}).
\end{align*}
Therefore we have
\begin{equation}
 \frac{\tau(\cq)L(s,f\times
\cq)\Gamma(s+(k-1)/2)}{(2\pi)^{s+(k-1)/2}}=C'\sum_{j=0}^{q-1}\chi_q(j)\int_{-c\log q}^{c\log q}
g_j(t)dt+O(q^{-\gamma}).
\end{equation}
Here $C'=q^{c|\re(s-\one)|}=\sup_{-c\log q\leq t\leq c\log q}|\exp(t(s-1/2)) | $ 
$$g_j(t)=\frac{1}{C'}\tf(n(j/q)a(t))\exp(t(s-1/2)).$$

Notice that $\frac{d}{dt}\mid_{t=t_0}(\tf(n(j/q)a(t)))=\frac{\partial}{\partial
x_2}\tf(n(j/q)a(t_0))
$. Here $\frac{\partial}{\partial x_2}$ denote the derivative of $\tf$ in the ``geodesic
direction'', defined in section \ref{notation}.  

We have assumed that $\tf $ has bounded derivatives
(ref. section \ref{notation}). Therefore using Leibnitz rule, for a fixed $s\in \mathbb C$ and each
$t$ in $[-c\log q,\log q ]$,
$$\frac{\partial^n}{\partial t^n}g_j(t_0)\ll_f (|s-1/2|+1)^n . $$ The constant involved is
independent of $q$. Hence for each $t$ in  $[-c\log q,c\log q ]$, $g_j$ is real analytic with radius
of convergence at least $1/(|s-1/2|+1) $. Therefore, choosing a
grid of $O(q^{\epsilon}) $ equispaced points and using power series expansion at the nearest grid
point on the left to compute $g_k$ at any given point, we get that given any $\gamma,\epsilon>0$,
\begin{align}
 &\int_{-c\log q}^{c\log q} \tf(n(j/q)a(t))\exp(t(s-1/2))dt\nonumber\\&=C'\sum_{x=-c
q^{\epsilon}\log q}^{cq^{\epsilon}(\log q)-1}\sum_{l=0}^{N'}\int_0^{q^{-\epsilon}}
\partial_2^l(g_j(xq^{-\epsilon}))\frac{t^l}{l !}dt+O(q^{-1-\gamma})\label{discretize}
\end{align}
Here $N'=O((1+\gamma)/\epsilon) $. Notice that the above equation is true for any $c>2$. Hence ,
given $\epsilon$, we can choose $c$ such that $\{ cq^{\epsilon}\log q\}=0 $.
 Multiplying \eqref{discretize} by $\chi_q(k) $ and summing over $k$, we get
\begin{align}
\label{approxfechihol1}
 L(s,f\times \cq)&=C''\sum_{x=-c q^{\epsilon}\log q}^{c q^{\epsilon}\log
q-1}\sum_{l=0}^{N'}d_l\sum_{j=0}^{q-1}\chi_q(j)\partial_2^l(g_j(xq^{-\epsilon}))+O(q^{-\gamma})
\end{align}
Here $d_l=\int_0^{q^{-\epsilon}} \frac{t^l}{l !}dt=\frac{q^{-\epsilon (l+1)}}{(l+1) !} $ and
$C''=\frac{C'(2\pi)^{s+(k-1)/2}}{\Gamma(s+(k-1)/2)\tcq} $.

\eqref{approxfechihol1}, along with the definition of $g_j(t)$ implies that the problem of computing
$L(s,f\times\chi) $ is equivalent to computing $\tcq$ and
$\sum_{j=0}^{q-1}\chi_q(j)\partial_2^l\tf(n(j/q)a(t)) $ for any $t\, \in \, [-c\log q,c\log q]$ and
for $0\leq l\leq N' $ faster than $O(q)$.  Here $\partial_2 $ denotes $\frac{\partial}{\partial x_2}
$. Exactly same method will work for the even Maass forms, when $\chi_q(-1)=-1 $.

Note that for an even Maass form if $\cq(-1)=-1$, then we need to compute the sum
$\sum_{j=0}^{q-1}\chi_q(j)\int_{q^{-c}}^{q^c} \partial_x f(j/q+iy)y^{s-1/2}dy $.
Notice that $\tf(n(x)a(\log y))=f(x+iy)$. Therefore, for any $x_0$, an easy computation gives 
$$\partial_xf(x_0+iy)=y^{-1}\frac{\partial}{\partial x_1}\tf(n(x_0)a(\log y)). $$
Using this equation, we can follow exactly the same method before, to get lemma when
$\chi_q(-1)=-1.$
\end{proof}
 In the following sections, we prove that $D(q)=(M^{5+\epsilon}+N)^{1+\epsilon} $. We use this
result to finish the proof of theorem \ref{main theorem 3}
\section{Proof of theorem \ref{main theorem 3}}
\label{q-aspect}
Let $\Gamma=\Slz$ and $f$ be a (holomorphic or Maass) cusp form of weight $k$ on $\gh$. Let $q=MN$
where $M\leq N$, $M=M_1M_2 $, where $M_1=\gcd(M,N)$ and $(M_2,N)=1 $. Let $\chi_q $ be a Dirichlet
character on $\mathbb Z/q\mathbb Z  $. Let $\gamma,\epsilon$ be any given positive numbers. The main
goal of this section is to prove theorem \ref{main theorem 3}.

Lemma \ref{integral to sum fe} implies that the algorithm to compute $L(s,f\times \chi_q) $
faster than $O(q)$ is equivalent to computing the Gauss sum $\tcq$ and
\begin{equation}\label{sumchi}\sum_{k=0}^{q-1}\chi_q(k)\partial_2^l\tf(n(k/q)a(t)) \end{equation}
 for any $t\, \in \, [-c\log q,c\log q]$ and for $0\leq l\leq N' $, faster than $O(q)$. Here
$\partial_2 $ denotes $\frac{\partial}{\partial x_2} $, $N'=O((1+\gamma)/\epsilon)$ and $c>2$ is a
suitable computable constant. Hence, in this section we will consider the problem of
computing the sum in \eqref{sumchi} for any given $t\in [-c\log q,c\log q]$.

\footnote{Note that for
a Maass form $f$, for $g=\tf$, $t=-\log q $ and $l=0 $, \eqref{sumchi} takes a more familiar form
\[\sum_{k=0}^{q-1}\chi_q(k)f(k/q+i/q).\] In this section, we essentially give an algorithm to
compute $\sum_{k=0}^{q-1}\chi_q(k)f(k/q+i/q)$, faster than $O(q)$.}

Computing $\tcq$ can be done very rapidly. It is easy to see that given any positive $\gamma$, we
can compute $\tcq $ up to error $O(q^{-\gamma})$, using at most $O(M^2+N)$ operations.

We will deal with computing \eqref{sumchi} in the special cases: $q=MN$
where $(M,N)=1 $ and the case when $M|N$ separately in the subsections \ref{MN} and \ref{Pk}
respectively. We use these results to complete the proof of theorem \ref{main theorem 3} in section
\ref{proof of main theorem 3}.

.
\subsection{Case $q=MN$ where $(M,N)=1 $}
\label{MN}
Let $g$ be a real analytic function with bounded derivatives on $\gsl$. Let $q=MN $ where $(M,N)=1
$. Let $t$ be a real number in $[-c\log q,c\log q] $,  $t_1=t+\log  q $. Let $$x_0=a(t_1). $$ In
this section, we will consider the problem of computing the sum 
\begin{align}\label{sumchi1}S&=\sum_{k=0}^{q-1}\chi_q(k)g(n(k/q)a(t))\nonumber\\&=
\sum_{k=0}^{q-1}\chi_q(k)g(A(q,1,k)a(t_1))\nonumber\\&=
\sum_{j=0}^{N-1}\sum_{k=0}^{M-1}\chi_q(j+Nk)g(A(q,1,j+kN)x_0).\end{align} Computing the sum of
this type is equivalent to computing the sum \eqref{sumchi}.
Recall that $A(q,1,k)$ is defined in section \ref{heckeorbits}. Use, $A(q,1,j+kN)=A(M,1,k)A(N,1,j) $
to rewrite \eqref{sumchi1} as

\begin{equation}
 \label{sumchi2}
 S=\sum_{j=0}^{N-1}S_{j}.
\end{equation}

Here, $S_j$ is defined as

\begin{equation}
 \label{sj}
S_j= \sum_{k=0}^{M-1}\cq(j+kN)g(A(M,1,k)v_j).
\end{equation}

Here $$v_j=A(N,1,j)x_0. $$ We now use $\cq=\chi_M\chi_N$ to get
\begin{align}
\label{sj1}
 S_j&= \sum_{k=0}^{M-1}\cq(j+kN)g(A(M,1,k)v_j)\nonumber\\&=
\chi_N(j)\sum_{k=0}^{M-1}\chi_M(j+kN)g(A(M,1,k)v_j).
\end{align}

Let us ``reduce'' the points $v_1,...,v_n $ to an approximate fundamental domain using the algorithm
in \cite[Chapter 7]{vishe}. Hence we have matrices $\{\gamma_j,x_j\} $ such that $ v_j=\gamma_jx_j
$, $\gamma_j\in \Slz$
 and $x_j$ lies in the ``approximate fundamental domain''.

$\chi_M$ is a character on $\mathbb Z/M\mathbb Z $. Given $j$, let $c_j$ be defined by
$j\equiv c_j \bmod M$ and $0\leq c_j<M$. Hence rewrite \eqref{sj1} as
\begin{equation}
 \label{sj2}
 S_j= \chi_N(j)\sum_{m|M}\sum_{k=0}^{M/m-1}h_{c_j}(m,k)g(A(M,m,k)\gamma_j x_j).
\end{equation}
Here for $\{0\leq l\leq M-1\}$, $h_l: T(M)\rightarrow \mathbb C^{\times} $ is defined as:
\begin{equation}
 \label{hbj}
h_{l}((m,k))=\delta_1(m)\chi_M(l+kN).
\end{equation}
Where, $\delta_1 $ is the characteristic function of $\{1\}$. For example, for a prime $M$
the functions $h_j $ are given explicitly by: $h_j(1,k)=\chi_M(j+kN)$ and $h_j(M,0)=0 $.

The number of points in $T(M) $ are at most $M^{1+\epsilon} $ ( using the fact that the number of
divisors of $M$ are at most $O(M^\epsilon)$). The right action by $a$ $\in$ $\Gamma$
permutes the $T(M)$ in $\gsl$. In other words, given any element $a$ of $\Gamma $, there exists a
permutation $\sigma_a $ on $T(M)$ such that for each $(m,k)\in T(M) $, we have an $a'\in \Gamma$
such that $$A(M,m,k)a=a'A(M,\sigma_a(m,k)).$$

Using lemma \ref{permutations} we get that if $a\equiv b\bmod M $, then $\sigma_a=\sigma_b $. This
implies that the total number of permutations of $T(M)$ due to the right action of $\Gamma$ is at
most $M^3 $. Using this, we rewrite \eqref{sj2} as
\begin{equation}
 \label{sj3}
 S_j= \chi_N(j)\sum_{(m,k)\,\in\, T(M)}h_{c_j}(\sigma_{\gamma_j}(m,k))g(A(M,m,k) x_j).
\end{equation}
Lemma \ref{permutations} implies that the number of distinct functions $h_{c_j}\sigma_{\gamma_j} $
is at most $M^4$. Let us enumerate them as $r_1,...,r_L $ (say) , where $L\leq M^4 $.

Notice that the Hecke orbits ``preserve'' the distance between the points. In other words, let
$M$ be any positive integer , and $x,y\in \Sl$ such that $x\in y V$ for some open neighbourhood
$V$
of the identity. Then $A(M,m,k)x $ lies in $A(M,m,k)y V $, for all $(m,k)\in T(M)$. Therefore, given
any
positive $\epsilon $, we ``sort'' the points $\{x_j\} $ into sets $K_1,...,K_{Q} $ such that $x,y\in
K_j $ implies that $x^{-1}y\,\in\,U_{q^{-\epsilon}} $. Here $Q=O(q^{3\epsilon}) $.\footnote{The
approximate fundamental domain has form $\{n(t)a(y)k(\theta):0\ll t\ll 1, 1\ll y\ll \log q, -\pi<
\theta\leq \pi\}$. We divide it into $Q$ ``cubes'' having sides $O(q^{-\epsilon})$ each. Every
$x_j$ lies in one of these $Q$ cubes. For each $j\leq Q$, $K_j$ consists of points $x_i$ is in the
$j^{\mathrm{th}}$ cube.} We choose a
representative $x_{n_j}$ from each $K_j$.

Let us focus on $K_1$. Let us use a power series expansion around the point $A(M,m,k)x_{n_1} $ to
compute $g(A(M,m,k)x) $ for all $x\,\in \, K_1 $. We summarize the result in the following lemma:
\begin{lemma}
 \label{epsilonchi}
Given any positive reals $\epsilon,\gamma $, positive integer $M$, $x$ and $y\,\in\, K_i $ for some
integer $i$. Let $r$ be any complex valued function defined on the the set $T(M)$. Then there are
explicitly computable constants $c_{\beta,x,y} $ and $d$ such that,
\begin{equation}
 J(0,r,y)=\sum_{|\beta|=0}^{d}c_{\beta,x,y}J(\beta,r,x) +O(M^2q^{-\gamma}).
\label{sum3}
\end{equation}
Here $d=O((1+\gamma)/\epsilon)$ and $J(\beta,r,x)$ is defined by
\begin{equation}
 \label{I}
J(\beta,r,x)=\sum_{(m,k)\in T(M)}r((m,k)) \partial^\beta g(A(M,m,k)x).
\end{equation}
The constant in $O$ is independent of $M,q$ and is a polynomial in $(1+\gamma)/\epsilon $.
\end{lemma}
We will prove the lemma \ref{epsilonchi} in section \ref{epsilonpower}.
 Notice that for fixed $\beta,r$ and a fixed $x$, we can compute $J(\beta,r,x)$ in time
$O(M^{1+\epsilon})$. This gives us an exact idea of the algorithm. The exact algorithm is given by:

\subsection{Explicit algorithm}
\begin{enumerate}
	\item \label{step2}Compute and store the functions $r_1,...,r_L$.
	\item \label{step3}Reduction of points $v_1,...,v_N$ into points $x_1,...,x_N$ and sorting
of the points into sets $K_1,...,K_Q$ and choose a representative $x_{n_i}$ from each $K_i$.
	\item \label{step4}For each $0\leq l<N$, compute $0\leq c_l\leq M-1$ be such that $c_l\equiv
l\bmod{M}$. Find $j_l $ such that $h_{c_l}
\sigma_{\gamma_l}=r_{j_l}$.
	\item \label{step5}$i\leftarrow 1$.
	\item \label{step6} for all $x_l$ in $K_i$, and for all $|\beta|\leq d$, compute and store
$c_{\beta,x_{n_i},x_l} $.
  \item \label{step7}for all $|\beta|\leq d $, and for all $1\leq t \leq L$,  compute and store
$J(\beta,r_{t},x_{n_i})$.
  \item \label{step8}for each $x_l$ in $K_i$, compute $$S_l=\chi_N(l)\sum_{|\beta|=0}^{d}
c_{\beta,x_{n_i},x_l}J(\beta,r_{j_l},x_{n_i}). $$
  \item \label{step9}if $i=Q$ compute $S=S_1+...+S_N$ else $i\leftarrow i+1$ and go to step
\ref{step6}.
	
\end{enumerate}
\subsection{Time complexity}
Let us compute the time complexity in the algorithm. The $O$ constants involved in this proof are
polynomial in $(1+\gamma)/\eta$.

 Computing and storing $r_i$ for all $i$ takes at most (roughly) $O(L M^{1+\epsilon})$
time. But recall that $L\leq M^4$. Hence step \ref{step2} takes $O(M^{5+\epsilon})$ time. 

It is easy to see that ``reducing'' each $v_i$ to an approximate fundamental domain, can be done in
$O(\log q)$ time. There are many standard reduction algorithms available. We refer the readers to
\cite{voight} for a form of the reduction algorithm. Cutting the approximate fundamental domain into
$Q=O(q^{3\epsilon})$ ``cubes'' and sorting the points $x_j$ in the sets $K_1,...,K_Q$, and choosing
a representative $x_{n_i}$ from each $K_i$ takes $O(N^{1+\epsilon}+q^{3\epsilon})$ time . Step
\ref{step4} also takes $O(N)$ time. The whole ``reducing'' and sorting process has
also been discussed in detail in \cite[chapter 7 ]{vishe}.

Notice that
for each $x_l\in S_i$, we need to compute $c_{\beta,x_{n_i},x_l} $ for $|\beta|\ll 1$. In section
\ref{epsilonpower}, it will be shown that each of these values can be computed in $O(1)$ time.
Therefore, step \ref{step6} takes $O(1)$ steps. Hence for all $l$, step \ref{step6} takes
$O(N)$ steps. Notice that for a fixed $i$, and fixed $t$, step \ref{step7} takes
$O(M^{1+\epsilon})$ time.

Hence for fixed $i$, and for all $t$, step \ref{step7} takes
$O(LM^{1+\epsilon}) $ time. Recall that $L\leq M^4 $. Therefore
total time spent in computing $J(\beta,r_{t},x_{n_i})$ for a fixed $i$ and all $t$ is
$O(M^{5+\epsilon}) $. Therefore, for all $i$, for all $j$, and for all $t$, the step \ref{step7}
takes $O(QM^{5+\epsilon})\approx O(M^{5+4\epsilon}) $ time. Recall here that $Q\approx
M^{3\epsilon} $ and $d=O(1)$. Steps \ref{step8} and \ref{step9} take $O(N) $
steps. 

Therefore, the total time spent is $O(M^{5+7\epsilon}+N) $.

\subsection{Case $q=MN$ when $M|N $}
\label{Pk}
Let $q=MN $ where $M|N$. Unless redefined in this section, we borrow the notations from
previous section.

We proceed as in previous section and rewrite \eqref{sumchi1} as

\begin{equation}
 \label{sumchi4}
 S=\sum_{j=0}^{N-1}S_{j}.
\end{equation}

Here, $S_j$ is defined as

\begin{equation}
 \label{sj'}
S_j= \sum_{k=0}^{M-1}\cq(j+kN)g(A(M,1,k)v_j).
\end{equation}
$v_j$ as in the previous section. For $(j,N)\neq 1$, $S_j=0 $. Therefore, rewrite
\eqref{sumchi4} as

\begin{equation}
 \label{sumchi5}
S_j= \starsum_{j\bmod{N}}S_j.
\end{equation}

Let $F(k)=\chi_q(1+kN)$. Since $F(k_1+k_2)=F(k_1)F(k_2)$, there exists a constant
$b$ such that
$F(k)=\cq(1+kN)=e(bk/M) $.
Using this we have that for any $j$, coprime to $N$,
\begin{equation}
 \label{sj'1}
S_j=\cq(j)\sum_{k=0}^{M-1}e(bj^{-1}k/M)g(A(M,1,k)v_j).
\end{equation}
Notice that $j^{-1} $ denotes the multiplicative inverse of $j\bmod m$. Let $c_j$ be defined by
$c_j\equiv j\bmod M$, where  $0\leq c_j\leq q-1$. We rewrite \eqref{sj'1} as
\begin{equation}
 \label{sj'2}
 S_j= \cq(j)\sum_{(m,k)\in T(M)}^{M/m-1}h_{c_j}((m,k))g(A(M,m,k)v_j).
\end{equation}
Here, for $0\leq l\leq M-1$, $h_l: T(M)\rightarrow \mathbb C^{\times} $
\begin{equation}
 \label{hbj'}
h_{l}((m,k))=\delta_1(m)e(bl^{-1} k/M ).
\end{equation}

The function $h_l$ only depends on the residue of $l\bmod M$. Therefore, there are $M$ distinct
functions $h_{c_j}$. We can proceed exactly similarly as in last section to
get the required algorithm.
\subsection{Proof of theorem \ref{main theorem 3}}
\label{proof of main theorem 3}
Let $q=MN$. Let $M=M_1M_2 $ such that $(M_2,N)=1$ and $M_1|N$. We proceed as in previous
sections and rewrite \eqref{sumchi1} as

\begin{equation}
 \label{sumchi'4}
 S=\sum_{j=0}^{N-1}S_{j}.
\end{equation}

Here, $S_j$ is defined as

\begin{equation}
 \label{sj''}
S_j= \sum_{k=0}^{M-1}\cq(j+kN)g(A(M,1,k)v_j).
\end{equation}
$v_j$ as in the previous sections. Notice that $\cq=\chi_{M_2}\chi_{M_1N} $ and that there exists
$b$ such that $\chi_{M_1N}(1+kN)=e(kb/M_1) $. Moreover, $S_j=0$, if $j$ is not coprime to $N$.
Therefore for $(j,N)=1$, rewrite \eqref{sj''} as
\begin{align}
&S_j=
\sum_{k=0}^{M_1-1}\sum_{l=0}^{M_2-1}\cq(j+(k+lM_1)N)g(A(M,1,k+lM_1)v_j)\nonumber\\&=\sum_{k=0}^{
M_1-1}\sum_{l=0}^{M_2-1}\cq(j+kN+lM_1N)g(A(M,1,k+lM_1)v_j)\nonumber\\&=\sum_{k=0}^{M_1-1}\chi_{M_1N}
(j+kN)\sum_{l=0}^{M_2-1}\chi_{M_2}(j+kN+lM_1N)g(A(M,1,k+lM_1)v_j)\nonumber\\&=\chi_{M_1N}(j)\sum_{
k=0}^{M_1-1}e(\frac{bj^{-1}k}{M_1})\sum_{l=0}^{M_2-1}\chi_{M_2}(j+kN+lM_1N)g(A(M,1,k+lM_1)v_j).
\label{sj'''}
\end{align}
Notice that for fixed $k$ and $l$, the quantity $e(bj^{-1}k/M_1) $ is uniquely determined for
$j\bmod{M_1} $ and $\chi_{M_2}(j+kN+lM_1N) $ is uniquely determined for $j\bmod{ M_2} $. Hence if
$j_1\equiv j_2 \bmod M $ then for all $0\leq k\leq M_1-1 $ and $0\leq l \leq M_2-1 $ we have that,
$$e(\frac{j_1^{-1}bk}{M_1})\chi_{M_2}(j_1+kN+lM_1N)=e(\frac{j_2^{-1}bk}{M_1})\chi_{M_2}
(j_2+kN+lM_1N). $$ Let $c_j$ be defined by $c_j\equiv j \bmod M $, where $0\leq c_j\leq M-1$.
We can
rewrite \eqref{sj'''} as

\begin{equation}
 \label{sj''''}
S_j=\chi_{M_1N}(j)\sum_{(m,k)\in T(M)}h_{c_j}((m,k))g(A(M,m,k)v_j).
\end{equation}

Here, for $(l,N)=1$, $h_{l}: T(M)\rightarrow \mathbb C^{\times} $ is defined by:
\begin{equation}
 \label{hbj''}
h_{l}((m,k))=\delta_1(m)e(b l^{-1} k/M_1 )\chi_{M_2}(l+k_0N+l_0M_1N);
\end{equation}
where, $$k=k_0+l_0M_1,\;0\leq k_0\leq M_1-1, \;0\leq l_0 \leq M_2-1.  $$

The inverse in \eqref{hbj''} is $\bmod{M_1}$. Notice again that there are only at most $M$ distinct
functions $h_{c_j}$ on $T(M)$. Therefore, the algorithm in \ref{MN} can be easily used to get an
algorithm in the general case. 

\section{Proof of lemmas \ref{epsilonchi} and \ref{permutations}}
\label{epsilonpower}
In this section we will give a proof of lemmas \ref{epsilonchi} and \ref{permutations}.
\subsection{proof of lemma \ref{epsilonchi}}
Let $x,y\in K_i$ for some integer $i$ and let $r$ be a function on the set $T(M) $. This implies
that
  $ x^{-1}y\in U_{q^{-\epsilon}}$. This implies that $$(A(M,m,k)x)^{-1}A(M,m,k)y=x^{-1}y\,\in\,
U_{q^{-\epsilon}} $$ for all $(m,k)\,\in\, T(M) $. Let $$x^{-1}y=n(t_0)a(y_0)K(\theta_0). $$ Here
$|\theta_0|\leq \pi$. As $g$ is a real analytic function on $\gsl $, we can use the power series
expansion for $g$ at points $A(M,m,k)x$ to compute $g(A(M,m,k)y)$ to get
\begin{equation}
 g(A(M,m,k)y)=\sum_{\beta=(\beta_1,\beta_2,\beta_3)\in\mathbb Z_+^3}\frac{\partial^\beta
g((A(M,m,k)x)}{\beta!}t_0^{\beta_1}y_0^{\beta_2}\theta_0^{\beta_3}.
\end{equation}
As $n(t_0)a(y_0)K(\theta_0)\in U_{q^{-\epsilon}}$, we get that there exists a constant
$d=O((1+\gamma)/\epsilon)$ such that
\begin{equation}
 g(A(M,m,k)y)=\sum_{|\beta|\leq d}\frac{\partial^\beta
g((A(M,m,k)x)}{\beta!}t_0^{\beta_1}y_0^{\beta_2}\theta_0^{\beta_3}+O(q^{-\gamma}).
\end{equation}
Substituting in \eqref{sj3}, we get
\begin{equation}
 \label{sj4}
J(0,r,y)=\sum_{|\beta|=0}^{d}c_{\beta,x,y}J(\beta,r,x) +O(2M^2q^{-\gamma}).
\end{equation}
Here $c_{\beta,x,y}=t_0^{\beta_1}y_0^{\beta_2}\theta_0^{\beta_3}/\beta ! $. Notice that each
$c_{\beta,l,n}$ can be computed in unit time.
\subsection{proof of lemma \ref{permutations}}
 Let $(m,k) $ be a any element of $T(L)$. Let $\Sgao(m,k)=(m_1,k_1) $. This implies that $\Gamma
A(L,m,k)a_1=\Gamma A(L,m_1,k_1) $. This implies that there exist a matrix $c\in \Slz$ such that
\begin{equation}\label{SGA}cA(L,m,k)a_1=A(L,m_1,k_1) .\end{equation} The lemma is equivalent is
proving that there exists a matrix $d$ in $\Slz$, such that
\begin{equation}\label{SGAD}dA(L,m,k)a_2=A(L,m_1,k_1)\end{equation}
 ( this would imply that $\Sgat(m,k)=\Sgao(m,k) $).

We use \eqref{SGA} to get \begin{align}&cA(L,m,k)a_2\nonumber\\&=c A(L,m,k)(a_1+Lb)\nonumber\\&=c
A(L,m,k)a_1 +cA(L,m,k)Lb\nonumber\\&= A(L,m_1,k_1)
+cA(L,m,k)Lb\nonumber\\&=A(L,m_1,k_1)+cA(L,m,k)LbA(L,m_1,k_1)^{-1}A(L,m_1,k_1)\nonumber\\&=(I+cA(L,m
,k)LbA(L,m_1,k_1)^{-1})A(L,m_1,k_1)\nonumber\\&=(I+cL^{-\one}\left(\begin {array}{cc}m&k\\0&L/m
\end{array} \right)LbL^{-\one}\left(\begin {array}{cc}L/m_1&-k_1\\0&L/m_1 \end{array}
\right))A(L,m_1,k_1)\nonumber\\&=(I+c\left(\begin {array}{cc}m&k\\0&L/m \end{array}
\right)b\left(\begin {array}{cc}L/m_1&-k_1\\0&L/m_1 \end{array}
\right))A(L,m_1,k_1).\label{SGADD}\end{align} 
\eqref{SGADD} implies that $d=c(I+c\left(\begin {array}{cc}m&k\\0&L/m \end{array}
\right)b\left(\begin {array}{cc}L/m_1&-k_1\\0&L/m_1 \end{array} \right) )^{-1} $ will be in $\Slz$
and will satisfy condition in equation \eqref{SGAD}. $(m,k) $ was any arbitrary element of $T(L)$.
This implies that $\Sgao=\Sgat $.
\bibliographystyle{amsalpha}\bibliographystyle{amsalpha}

\end{document}